\documentclass[11pt,a4paper]{article}

\usepackage[margin=1.1in]{geometry}
\usepackage{newpxtext}
\usepackage{amsmath,amssymb,amsthm,mathtools}

\usepackage{newpxmath}
\usepackage{microtype}

\usepackage{enumitem}
\usepackage{booktabs}
\usepackage{xcolor}
\usepackage{titlesec}
\usepackage{fancyhdr}
\usepackage{hyperref}

\definecolor{journalblue}{RGB}{16,76,128}

\hypersetup{
  colorlinks=true,
  linkcolor=journalblue,
  citecolor=journalblue,
  urlcolor=journalblue,
  pdfstartview=FitH
}

\titleformat{\section}
  {\normalfont\large\bfseries\color{journalblue}}
  {\thesection.}
  {0.6em}
  {}
\titleformat{\subsection}
  {\normalfont\bfseries\color{journalblue}}
  {\thesubsection.}
  {0.5em}
  {}

\setlist[enumerate]{noitemsep, topsep=4pt, parsep=0pt, partopsep=0pt}
\setlist[itemize]{noitemsep, topsep=4pt, parsep=0pt, partopsep=0pt}

\newcommand{\Q}{\mathbb{Q}}

\newcommand{\Z}{\mathbb{Z}}
\newcommand{\Pbb}{\mathbb{P}}

\newcommand{\Gal}{\operatorname{Gal}}

\newcommand{\gon}{\operatorname{gon}}

\newcommand{\aairr}{\operatorname{a.irr}}

\theoremstyle{plain}
\newtheorem{theorem}{Theorem}[section]
\newtheorem{proposition}[theorem]{Proposition}
\newtheorem{lemma}[theorem]{Lemma}
\newtheorem{corollary}[theorem]{Corollary}

\theoremstyle{definition}
\newtheorem{definition}[theorem]{Definition}

\theoremstyle{remark}
\newtheorem{remark}[theorem]{Remark}

\title{\vspace{-1.5em}\LARGE\bfseries\color{journalblue}
Cubic Progressions on Smooth Plane Curves}

\author{\textbf{Eslam Badr}}
\date{}

\begin{document}

\maketitle
\thispagestyle{empty}

\vspace{-1em}
\begin{center}
\begin{minipage}{0.92\textwidth}
\small
\rule{\textwidth}{0.4pt}\\[0.6em]
\textbf{\color{journalblue} Abstract.}
Let $C \subset \Pbb^2_k$ be a smooth projective plane curve of
degree $d \ge 5$ over a number field $k$. We prove that any cubic
arithmetic or geometric progression sequence on $C$ is finite. The
proof rests on a prime-$p$ tower of cyclic Kummer covers controlled
by the divisor of a linear-form quotient on $C$, combined with the
Kadets--Vogt classification of curves with infinitely many cubic
points, the Abramovich--Harris gonality descent, a
Castelnuovo--Severi factorization for the elliptic maps supplied by
Kadets--Vogt independently at consecutive levels of the tower, and a
deck-transformation obstruction. We also obtain a conditional
quartic result under a cyclotomic non-splitting hypothesis. The
case $d = 4$ with $C(k) \neq \emptyset$ admits infinite cubic
progressions, so no analogue of the main theorem can hold there.\\[0.6em]
\rule{\textwidth}{0.4pt}
\end{minipage}
\end{center}
\vspace{1em}

\section{Introduction}
A classical theme in Diophantine geometry is the interplay between
the algebraic structure of a curve and the arithmetic of special
subsets of its rational points. A particularly rigid such subset is
furnished by requiring that one coordinate of the points vary in a
prescribed elementary sequence. Concretely, if $C$ is a smooth
plane curve over a number field $k$, one may ask how many of its
points of bounded degree can have $x$-coordinates lying in an
arithmetic or geometric progression. This question goes back to
work of Mohanty \cite{Mohanty1975} and Bremner
\cite{Bremner1999} on elliptic curves, with subsequent
contributions by Campbell \cite{Campbell2003}, MacLeod
\cite{MacLeod2006}, Ulas \cite{Ulas2005}, and Alaa--Sadek
\cite{AlaaSadek2016}, among others.

\subsection*{The quadratic case}

For curves of genus $\ge 2$, Faltings' theorem \cite{Faltings1983}
immediately implies that rational progressions are finite. The
natural next case is that of \emph{quadratic} progressions. For
hyperelliptic curves, taking square roots of the value of the
defining polynomial yields an infinite quadratic progression. For
a general curve, the set of quadratic points may be infinite even
without the curve being hyperelliptic, namely when the curve is
bielliptic with an associated elliptic curve of positive rank; a
theorem of Abramovich--Harris \cite{AbramovichHarris1991} and
Harris--Silverman \cite{HarrisSilverman1991} makes this dichotomy
precise. In \cite{BadrSadek2020}, Badr and Sadek showed that the
bielliptic case, despite admitting infinitely many quadratic
points, does \emph{not} give rise to infinite quadratic
progressions on smooth plane quartics: a tower of double covers
leads via Castelnuovo--Severi to a contradiction.

\subsection*{The cubic case}

The present note addresses the next case, namely \emph{cubic}
progressions. Our main result is the following.

\begin{theorem}\label{thm:main}
Let $C$ be a smooth projective plane curve of degree $d \ge 5$
over a number field $k$. Then any cubic geometric (resp.\
arithmetic) progression sequence on $C$ is finite.
\end{theorem}

For smooth plane curves of degree $d \ge 5$, combining Faltings'
finiteness of rational points, the results of \cite{BadrSadek2020}
on quadratic progressions, and Theorem~\ref{thm:main} on cubic
progressions settles the finiteness question for rational,
quadratic, and cubic progressions. For $d = 4$ the situation is
more delicate and depends on the rational points of $C$: when
$C(k) \neq \emptyset$, the curve is trigonal over $k$ and by
\cite[Thm.~5.2]{BadrSadek2020} admits infinite cubic progressions,
so no analogue of Theorem~\ref{thm:main} can hold there; when
$C(k) = \emptyset$, a variant of the tower argument yields the
same finiteness statement over $k$ under an additional cyclotomic
non-splitting hypothesis; see Proposition~\ref{prop:d4}.

The main new ingredient is a \emph{prime-$p$ tower}. Instead of the
double cover tower of \cite{BadrSadek2020}, we construct, for a
suitable prime $p > 3$ with $p \nmid d(d-1)$, and over the
cyclotomic extension $k(\zeta_p)$, a tower of degree-$p$ cyclic
Kummer covers $X_n \to X_{n-1} \to \cdots \to C$ whose ramification
is controlled by the divisor of a linear-form quotient $f$ on $C$.
Since $\deg(f) \in \{d-1, d\}$ and $p \nmid d(d-1)$, at least one
zero of $f$ has $p$-valuation prime to $p$. This guarantees a
totally ramified degree-$p$ Kummer tower and supplies the fixed
point needed for the final deck-transformation argument. The
construction also automatically covers both the bielliptic and
trielliptic outputs of the Kadets--Vogt classification.

\subsection*{Structure of the paper}

Section~\ref{sec:prelim} collects background: progressions,
gonality, the Kadets--Vogt classification of curves with infinitely
many cubic points, the Abramovich--Harris gonality descent, and the
Castelnuovo--Severi inequality. Section~\ref{sec:quotient}
introduces the linear-form quotient and establishes the existence
of a point with $p$-coprime valuation. Section~\ref{sec:tower}
constructs the prime-$p$ tower and proves it is a genuine tower of
degree-$p$ totally ramified covers, with geometrically integral
members. Section~\ref{sec:gonality} proves gonality descent and
Section~\ref{sec:genus} establishes genus growth.
Section~\ref{sec:application} applies the Kadets--Vogt
classification and the Castelnuovo--Severi factorization lemma to
show that, at consecutive levels of the tower, the elliptic maps
must be isogeny-related, and Section~\ref{sec:descent} completes the
proof via the deck-transformation descent.
Section~\ref{sec:arithmetic} treats the arithmetic case.
Section~\ref{sec:example} illustrates the tower on the quintic
$X^5 + Y Z^4 + Y^4 Z = 0$. Section~\ref{sec:empty} discusses
$d = 4$ and the case $C(k) = \emptyset$. Section~\ref{sec:concl}
records open problems.

\section{Preliminaries}\label{sec:prelim}

\subsection{Progressions on plane curves}

Let $C \subset \Pbb^2_k$ be a projective plane curve given by a
homogeneous polynomial $F(X,Y,Z)$ with coefficients in a number
field $k$, and let $\bar{k}$ be a fixed algebraic closure of $k$.
For $n \ge 1$, denote by $\Gamma_n(C,k)$ the set of $\bar{k}$-points
of $C$ whose field of definition over $k$ has degree at most $n$.

\begin{definition}\label{def:progression}
An \emph{$n$-level geometric progression sequence} on $C$ is a
sequence of points $P_i = (x_i : y_i : z_i) \in \Gamma_n(C,k)$,
$i \ge 1$, for which there exist $t, t' \in k^\ast$ such that
$x_i/y_i = t' t^i$ and $y_i \in k$ for every $i$. Replacing the
condition on $x_i/y_i$ by $x_i/y_i = t' + i t$ yields the notion of
an \emph{$n$-level arithmetic progression sequence}. In both cases
we require the $x_i/y_i$ to be pairwise distinct; equivalently, for
a geometric progression, $t$ is not a root of unity.
\end{definition}

After a harmless change of variables one may always assume $Y = 1$,
so that each $P_i$ is written $(x_i : 1 : z_i)$ with $x_i \in k$
and $z_i$ of degree $\le n$ over $k$. This is the normalisation
used throughout.

\subsection{Gonality and its arithmetic counterpart}
\label{sec:gonality-prelim}

The \emph{geometric gonality} $\gon(C)$ of a smooth projective
curve $C$ over an algebraically closed field is the least degree of
a non-constant morphism $C \to \Pbb^1$. When $C$ is defined over a
number field $k$, the \emph{arithmetic gonality} $\gon(C,k)$ is the
least degree of a non-constant $k$-morphism $C \to \Pbb^1_k$;
clearly $\gon(C,k) \ge \gon(C_{\bar{k}})$.

For a smooth plane curve of degree $d$, the geometric gonality is
well known to be $d - 1$, and every gonal map is obtained by
projection from a point of the curve (see
\cite{Namba1979,Yoshihara2001}). We record the arithmetic version
of this fact, which will be used repeatedly.

\begin{lemma}\label{lem:gonality-plane}
Let $C \subset \Pbb^2_k$ be a smooth plane curve of degree $d \ge 4$
over a number field $k$. Then
\[
\gon(C,k) \;=\;
\begin{cases}
d - 1, & \text{if } C(k) \neq \emptyset, \\
d, & \text{if } C(k) = \emptyset.
\end{cases}
\]
\end{lemma}

\begin{proof}
If $C(k) \neq \emptyset$, projecting from a $k$-rational point of
$C$ realises a degree-$(d-1)$ $k$-morphism $C \to \Pbb^1_k$, so
$\gon(C,k) \le d-1$. Since the geometric gonality of $C_{\bar{k}}$
is $d-1$, no smaller degree is possible over any field, hence
$\gon(C,k) = d-1$.

If $C(k) = \emptyset$, projecting from a $k$-point
$P \in \Pbb^2(k) \setminus C$ gives a degree-$d$ $k$-morphism
$C \to \Pbb^1_k$, so $\gon(C,k) \le d$. For the reverse inequality,
suppose $\gon(C,k) \le d-1$. Then $C_{\bar{k}}$ admits a
$\bar{k}$-morphism of degree $\le d-1$ to $\Pbb^1_{\bar{k}}$.
Since the geometric gonality of $C_{\bar{k}}$ is $d-1$, such a
morphism has degree exactly $d-1$, hence is a $g^1_{d-1}$ on
$C_{\bar{k}}$ and, by the classical theory of gonal maps on smooth
plane curves, is the projection from a point $Q \in C(\bar{k})$.
The corresponding $g^1_{d-1}$ is defined over $k$, and its
projection center is unique; hence $Q$ is fixed by
$\Gal(\bar{k}/k)$, i.e.\ $Q \in C(k)$, contradicting
$C(k) = \emptyset$. Therefore $\gon(C,k) \ge d$.
\end{proof}

\subsection{The Kadets--Vogt classification}

The structural input for the proof is the arithmetic classification
of curves with infinitely many cubic points, due to Kadets and Vogt
\cite{KadetsVogt2025}. Following their formulation, we define the
\emph{arithmetic degree of irrationality} of a nice curve $X$ over
$k$ as the smallest integer $d \ge 1$ such that $X$ has infinitely
many closed points of degree $d$ over $k$:
\[
\aairr_k(X) \;:=\;
\min\bigl\{d \ge 1 : \text{$X$ has infinitely many closed points
of degree $d$ over $k$}\bigr\}.
\]

\begin{theorem}[Kadets--Vogt]\label{thm:KV-full}
Let $X$ be a nice curve over a number field $k$ (smooth, projective,
geometrically integral). Then:
\begin{enumerate}[label=(\roman*)]
\item If $\aairr_k(X) = 2$, then $X$ admits a $k$-morphism of
degree $2$ to $\Pbb^1_k$ or to an elliptic curve of positive rank
over $k$.
\item If $\aairr_k(X) = 3$, then either:
\begin{enumerate}[label=(\alph*)]
\item $X$ admits a $k$-morphism of degree $3$ to $\Pbb^1_k$ or to
an elliptic curve of positive rank over $k$; or
\item $X$ is a smooth plane quartic with no $k$-rational points,
positive-rank Jacobian, and at least one closed point of degree
$3$ over $k$; or
\item $X$ is one of the genus-$4$ Debarre--Fahlaoui curves
\cite{DebarreFahlaoui1993}.
\end{enumerate}
\end{enumerate}
\end{theorem}

\begin{remark}\label{rem:KV-form}
We use the theorem only in the form of Corollary~\ref{cor:KV}
below, where the genus hypothesis $g \ge 5$ excludes cases (b) and
(c).
\end{remark}

\begin{corollary}\label{cor:KV}
Let $X$ be a smooth projective geometrically integral curve of
genus $g \ge 5$ over a number field $k$. Suppose that
$\gon(X,k) \ge 4$ and that $\Gamma_3(X,k)$ is infinite. Then $X$ is
either bielliptic or trielliptic over $k$: that is, $X$ admits a
$k$-morphism of degree $2$ or of degree $3$ to an elliptic curve
of positive rank over $k$.
\end{corollary}

\begin{proof}
Since $g(X) \ge 5 \ge 2$, Faltings' theorem \cite{Faltings1983}
gives that $X(k)$ is finite. Since $\Gamma_3(X,k)$ is infinite,
there must therefore be infinitely many closed points of degree
exactly $2$ or exactly $3$ over $k$. Hence
$\aairr_k(X) \in \{2, 3\}$.

If $\aairr_k(X) = 2$, Theorem~\ref{thm:KV-full}(i) gives a
$k$-morphism of degree $2$ to $\Pbb^1_k$ or to an elliptic curve of
positive rank. The first is excluded by $\gon(X,k) \ge 4$. Hence
$X$ is bielliptic.

If $\aairr_k(X) = 3$, Theorem~\ref{thm:KV-full}(ii) applies.
Alternative (a) is either a degree-$3$ map to $\Pbb^1_k$, excluded
by $\gon(X,k) \ge 4$, or a degree-$3$ map to a positive-rank
elliptic curve, which is the desired conclusion. Alternatives (b)
and (c) have genus $3$ and $4$ respectively, both contradicting
$g \ge 5$. Hence $X$ is trielliptic.
\end{proof}

\begin{remark}\label{rem:KV-exceptions}
The genus threshold $g \ge 5$ in Corollary~\ref{cor:KV} is exactly
what removes the exceptional cases appearing in the general
Kadets--Vogt classification. For our application the tower curves
will have genus $> 19$, so the hypotheses are satisfied with room
to spare.
\end{remark}

\subsection{The Abramovich--Harris gonality descent}

We shall use the following elementary descent statement, which is
the $S(m,0)$ statement of \cite[§2]{AbramovichHarris1991}.

\begin{lemma}[AH gonality descent]\label{lem:AH-descent}
Let $\pi: X' \to X$ be a non-constant $k$-morphism of smooth
projective curves over $k$, and let $m \ge 1$. If $X'$ admits a
$k$-morphism of degree at most $m$ to $\Pbb^1_k$, then $X$
admits a $k$-morphism of degree at most $m$ to $\Pbb^1_k$.
\end{lemma}

In particular, if $\gon(X,k) \ge m+1$ and $X' \to X$ is
non-constant, then $\gon(X',k) \ge m+1$.

\subsection{The Castelnuovo--Severi inequality}

The following inequality is classical; a convenient reference for
the version we use is \cite[§2]{Accola1994}.

\begin{proposition}[Castelnuovo--Severi]\label{prop:CS}
Let $X$ be a smooth projective curve of genus $g$, and let
$f_1: X \to C_1$ and $f_2: X \to C_2$ be non-constant morphisms of
degrees $d_1$ and $d_2$ onto smooth projective curves of genera
$g_1$ and $g_2$ respectively. If the induced morphism
$X \to C_1 \times C_2$ is birational onto its image, then
\[
g \le d_1 g_1 + d_2 g_2 + (d_1 - 1)(d_2 - 1).
\]
\end{proposition}

\begin{corollary}\label{cor:CS-ell}
Let $X$ be a smooth projective curve of genus $g$, and let
$f_1: X \to E_1$ and $f_2: X \to E_2$ be non-constant morphisms of
degrees $d_1$ and $d_2$ to elliptic curves. If
$g > d_1 + d_2 + (d_1-1)(d_2-1)$, then the product morphism
$X \to E_1 \times E_2$ is not birational onto its image.
\end{corollary}

\section{The linear-form quotient and its divisor}
\label{sec:quotient}

Let $C \subset \Pbb^2_k$ be a smooth plane curve of degree $d \ge 4$
with $\gon(C,k) \ge 4$. Suppose, for contradiction, that $C$
admits an infinite cubic geometric progression sequence
$P_i = (x_i : 1 : z_i) \in \Gamma_3(C,k)$ with $x_i = t' t^i$ and
the $x_i$ pairwise distinct. The arithmetic case is handled
analogously in \S\ref{sec:arithmetic}.

Define
\[
f \;=\; \frac{X}{t' Y} \;\in\; k(C)^\ast,
\]
so that $f(P_i) = t^i$ for every $i$.

\begin{lemma}\label{lem:deg-f}
The morphism $f: C \to \Pbb^1_k$ induced by the quotient of the two
linear forms $X$ and $t'Y$ has degree
\[
\deg(f) \;=\; \begin{cases} d - 1, & \text{if } (0:0:1) \in C, \\
d, & \text{otherwise.} \end{cases}
\]
In particular, $\deg(f) \ge d - 1 \ge 3$.
\end{lemma}

\begin{proof}
The map $f$ is the projection of $C$ from the common zero
$(0:0:1)$ of the two linear forms $X$ and $Y$. A projection of a
smooth plane curve of degree $d$ from a point of $\Pbb^2$ has
degree $d - 1$ if the point lies on $C$ and degree $d$ otherwise.
\end{proof}

\begin{lemma}\label{lem:exists-odd}
Let $p$ be a prime with $p \nmid \deg(f)$. Then there exists a point
$P \in C(\bar{k})$ with $v_P(f)$ not divisible by $p$.
\end{lemma}

\begin{proof}
The divisor of zeros of $f$ has degree $\deg(f)$. If $p \nmid
\deg(f)$, then at least one zero has $p$-valuation not divisible
by $p$; otherwise the total degree $\deg(f)$ would be a sum of
$p$-multiples, hence divisible by $p$. Such a point has
$v_P(f) > 0$ not divisible by $p$.
\end{proof}

Since $\deg(f) \in \{d-1, d\}$, the hypothesis $p \nmid d(d-1)$
ensures $p \nmid \deg(f)$.

\section{The prime-$p$ tower}
\label{sec:tower}

Fix a prime $p > 3$ with $p \nmid d(d-1)$. Replacing $k$ by
$k(\zeta_p)$, we henceforth assume $\zeta_p \in k$. This is
harmless: $\Gamma_3(C, k) \subseteq \Gamma_3(C, k(\zeta_p))$, so an
infinite progression over $k$ is also an infinite progression over
$k(\zeta_p)$; and if finiteness holds over $k(\zeta_p)$, it holds
over $k$. From this point onward $k$ denotes the enlarged field.

Let $K = k(C)$ be the function field of $C$ over $k$. Put
$u_0 = f$ and $K_0 = K$. Recursively choose $u_{n+1}$ with
$u_{n+1}^p = u_n$ and set
\[
K_{n+1} \;=\; K_n(u_{n+1}).
\]
Then $u_n^{\,p^n} = f$ for every $n \ge 0$. Let $X_n$ be the smooth
projective curve over $k$ with function field $K_n$. The inclusion
$K \hookrightarrow K_n$ induces a $k$-morphism
$\phi_n: X_n \to C$ of degree $p^n$.

Choose once and for all a point $P_0 \in C(\bar{k})$ with
$p \nmid v_{P_0}(f)$; such a point exists by
Lemma~\ref{lem:exists-odd}.

\begin{lemma}\label{lem:degree-p}
For every $n \ge 0$, the following hold.
\begin{enumerate}[label=(\roman*)]
\item The extension $K_{n+1}/K_n$ has degree $p$. The induced
morphism $\phi_{n+1, n}: X_{n+1} \to X_n$ is a separable degree-$p$
cover of smooth projective curves over $k$, totally ramified at the
unique point above $P_0$.
\item $X_n$ is geometrically integral over $k$.
\end{enumerate}
\end{lemma}

\begin{proof}
We first establish (i). We show by induction on $n$ that:
\begin{itemize}
\item[(a)] the extension $K_n/K$ is totally ramified above $P_0$;
\item[(b)] for the unique point $P_n \in X_n(\bar{k})$ above $P_0$,
the valuation $v_{P_n}(u_n)$ is not divisible by $p$.
\end{itemize}
For $n = 0$: (a) is trivial, and (b) is the choice of $P_0$.

Assume (a) and (b) for some $n$. The extension
$K_{n+1} = K_n(u_{n+1})$, with $u_{n+1}^p = u_n$, is a Kummer
extension (since $\zeta_p \in k \subseteq K_n$). Its degree is $p$
if $u_n$ is not a $p$-th power in $K_n$, and $1$ otherwise.
Since $v_{P_n}(u_n)$ is not divisible by $p$, $u_n$ cannot be a
$p$-th power in $K_n$: if $u_n = w^p$ with $w \in K_n^\ast$, then
$p \, v_{P_n}(w) = v_{P_n}(u_n)$, contradicting
$p \nmid v_{P_n}(u_n)$. Hence $[K_{n+1} : K_n] = p$. Since
$p \nmid v_{P_n}(u_n)$, the Kummer extension is totally ramified
at $P_n$, so there is a unique point $P_{n+1} \in X_{n+1}(\bar{k})$
above $P_n$, with $v_{P_{n+1}}(u_{n+1}) = v_{P_n}(u_n)$, which is
not divisible by $p$. This gives (a) and (b) for $n+1$ and
completes the induction. Statement (i) follows.

For (ii), we argue simultaneously over $\bar{k}$. Put
$\bar{K}_0 = \bar{k}(C)$. We claim by induction on $n$ that
$\bar{K}_n := K_n \otimes_k \bar{k}$ is a field and that
$[\bar{K}_n : \bar{k}(C)] = p^n$.

The base case $n = 0$ is trivial. Suppose the claim holds for
some $n$, and denote by $P_n$ the unique point of $X_n(\bar{k})$
above $P_0$; the ramification argument above, applied over
$\bar{k}$, gives $p \nmid v_{P_n}(u_n)$. Hence $u_n \notin
\bar{K}_n^{\,p}$: if $u_n = w^p$ with $w \in \bar{K}_n^\ast$,
then $p \, v_{P_n}(w) = v_{P_n}(u_n)$, contradicting
$p \nmid v_{P_n}(u_n)$. Therefore the polynomial $T^p - u_n$ is
irreducible over $\bar{K}_n$, so
\[
[\bar{K}_{n+1} : \bar{K}_n]
= [\bar{K}_n(u_{n+1}) : \bar{K}_n]
= p.
\]
In particular, $\bar{K}_{n+1}$ is a field, and
$[\bar{K}_{n+1} : \bar{k}(C)] = p^{n+1}$.

Iterating, we obtain $[\bar{K}_n : \bar{k}(C)] = p^n$ for every
$n$. In particular, $\bar{K}_n$ is a field, so
$K_n \otimes_k \bar{k} = \bar{K}_n$ is a field and $X_n$ is
geometrically integral over $k$.
\end{proof}

\begin{corollary}\label{cor:ramified}
For every $n \ge 0$, the degree-$p$ cover
$\phi_{n+1, n}: X_{n+1} \to X_n$ is ramified. Consequently, the
deck transformation $\sigma_{n+1}$ (a generator of
$\Gal(K_{n+1}/K_n) \cong \Z/p$) has a fixed point on $X_{n+1}$.
\end{corollary}

\begin{proof}
By Lemma~\ref{lem:degree-p}, the cover is totally ramified at the
unique point above $P_0$, so it is ramified. The unique point
above $P_0$ is fixed by $\sigma_{n+1}$.
\end{proof}

\begin{lemma}\label{lem:tower-points}
For every $n \ge 0$, the set $\Gamma_3(X_n, k)$ is infinite.
\end{lemma}

\begin{proof}
Consider the subsequence of the progression indexed by
$i = r p^n$, $r = 1, 2, 3, \ldots$. The corresponding point of $C$
is $P_{r p^n} = (t' t^{r p^n} : 1 : z_{r p^n})$, which lies in
$\Gamma_3(C, k)$ by hypothesis. Its $x$-coordinate
$x_{r p^n} = t' t^{r p^n}$ lies in $k$, and
\[
f(P_{r p^n}) \;=\; \frac{x_{r p^n}}{t'} \;=\; t^{r p^n}
\;=\; (t^r)^{p^n}.
\]

Fix $r \ge 1$. Since $t^r \in k(P_{r p^n})$ satisfies
$T^{p^n} - f(P_{r p^n}) = 0$, there is a point
$Q_r \in X_n(\bar{k})$ lying above $P_{r p^n}$ with residue
$u_n(Q_r) = t^r$; such a point exists on the normalization
because $u_n^{\,p^n} = f$ in $K_n^\ast$ and the value $t^r$ is a
root of $T^{p^n} - f(P_{r p^n})$ in the residue field of
$P_{r p^n}$.

The residue field of $Q_r$ contains $u_n(Q_r) = t^r \in k$ and
the image $P_{r p^n}$, and is contained in the compositum
$k(P_{r p^n})(u_n(Q_r))$. Since $u_n(Q_r) \in k \subseteq
k(P_{r p^n})$, we have
\[
k(Q_r) \;=\; k(P_{r p^n}),
\]
and hence
\[
[k(Q_r) : k] \;=\; [k(P_{r p^n}) : k] \;\le\; 3.
\]
Thus $Q_r \in \Gamma_3(X_n, k)$.

The points $Q_r$, $r \ge 1$, are pairwise distinct because their
images $P_{r p^n}$ on $C$ are distinct. Hence $\Gamma_3(X_n, k)$
is infinite.
\end{proof}

\begin{remark}
The argument shows more precisely that the map $Q_r \mapsto
P_{r p^n}$ preserves field of definition: the lift introduces no
additional extension because the value of $u_n$ is already in $k$.
\end{remark}

\section{Gonality of the tower curves}
\label{sec:gonality}

\begin{lemma}\label{lem:gonality-tower}
For every $n \ge 0$, we have $\gon(X_n, k) \ge 4$.
\end{lemma}

\begin{proof}
The composed morphism $\phi_n: X_n \to C$ is non-constant. Suppose
$\gon(X_n, k) \le 3$. Then $X_n$ admits a $k$-morphism of degree
$\le 3$ to $\Pbb^1_k$, and Lemma~\ref{lem:AH-descent} with $m = 3$
yields a $k$-morphism $C \to \Pbb^1_k$ of degree $\le 3$,
contradicting $\gon(C, k) \ge 4$.
\end{proof}

\section{Genus growth}
\label{sec:genus}

\begin{lemma}\label{lem:genus-growth}
The genus $g(X_n)$ tends to infinity as $n \to \infty$. In
particular, for any fixed constant $M$ there exists $n_0$ such
that $g(X_n) > M$ for all $n \ge n_0$.
\end{lemma}

\begin{proof}
By Lemma~\ref{lem:degree-p}, each $\phi_{n+1, n}: X_{n+1} \to X_n$
is a separable degree-$p$ morphism of smooth projective curves.
Riemann--Hurwitz gives
\[
2 g(X_{n+1}) - 2 \;=\; p \bigl(2 g(X_n) - 2\bigr) + \deg(R_n),
\]
where $R_n$ is the ramification divisor of $\phi_{n+1, n}$. By
Corollary~\ref{cor:ramified}, the cover is totally ramified at the
unique point above $P_0$, so $\deg(R_n) \ge p - 1 \ge 4$, and in
particular
\[
g(X_{n+1}) - 1 \;\ge\; p \bigl(g(X_n) - 1\bigr).
\]
Iterating from $g(X_0) = g(C) = (d-1)(d-2)/2 \ge 3$ gives
\[
g(X_n) \;\ge\; p^n (g(C) - 1) + 1 \;\to\; \infty.
\]
\end{proof}

\section{The bielliptic/trielliptic alternative and the
Castelnuovo--Severi factorization}
\label{sec:application}

Fix $n$ large enough that
\begin{equation}\label{eq:size}
g(X_n) > 9p + 1
\qquad\text{and}\qquad
g(X_{n-1}) \ge 5.
\end{equation}
Such $n$ exists by Lemma~\ref{lem:genus-growth}.

By Lemma~\ref{lem:tower-points}, $\Gamma_3(X_n, k)$ is infinite,
and by Lemma~\ref{lem:gonality-tower}, $\gon(X_n, k) \ge 4$. By
Corollary~\ref{cor:KV}, the curve $X_n$ is bielliptic or
trielliptic over $k$. Let
\[
j_n: X_n \to E_n
\]
be a $k$-morphism of degree $m \in \{2, 3\}$ to an elliptic curve
$E_n$ of positive rank over $k$. Similarly, $X_{n-1}$ is
bielliptic or trielliptic over $k$; let
\[
j_{n-1}: X_{n-1} \to E_{n-1}
\]
be a $k$-morphism of degree $m' \in \{2, 3\}$ to an elliptic curve
$E_{n-1}$ of positive rank over $k$.

Let $\phi := \phi_{n, n-1}: X_n \to X_{n-1}$ be the degree-$p$
morphism from Lemma~\ref{lem:degree-p}. Define
\[
h_n \;:=\; j_{n-1} \circ \phi: X_n \to E_{n-1},
\]
which is a $k$-morphism of degree $p m'$.

\begin{proposition}\label{prop:factorization}
With notation as above, and $g(X_n) > 9p + 1$:
\begin{enumerate}[label=(\roman*)]
\item If $m \neq m'$, we obtain a contradiction.
\item If $m = m'$, there exists a $k$-morphism
$\psi_n: E_n \to E_{n-1}$ of degree $p$, which is a translate of
an isogeny, such that
\[
h_n \;=\; \psi_n \circ j_n.
\]
\end{enumerate}
\end{proposition}

\begin{proof}
Consider the product morphism
$(j_n, h_n): X_n \to E_n \times E_{n-1}$. If it were birational
onto its image, Corollary~\ref{cor:CS-ell} with $d_1 = m$,
$d_2 = pm'$ would give
\[
g(X_n) \;\le\; m + pm' + (m-1)(pm' - 1).
\]
For $(m, m') \in \{2, 3\}^2$, the right-hand side is at most
$9p + 1$ (achieved at $m = m' = 3$), contradicting
$g(X_n) > 9p + 1$.

Hence $(j_n, h_n)$ is not birational onto its image. Let
$\Gamma \subset E_n \times E_{n-1}$ be the image of $(j_n, h_n)$,
an integral curve, and let $D$ be its normalization. Then
$(j_n, h_n)$ factors through the canonical morphism
$\alpha: X_n \to D$ and the normalization morphism
$\beta: D \to E_n \times E_{n-1}$. Since $(j_n, h_n)$ is not
birational, $\deg(\alpha) > 1$.

Write $\beta = (\beta_1, \beta_2)$ with $\beta_1: D \to E_n$ and
$\beta_2: D \to E_{n-1}$. Then
\[
j_n = \beta_1 \circ \alpha, \qquad h_n = \beta_2 \circ \alpha.
\]
Taking degrees,
\[
m = \deg(\beta_1) \cdot \deg(\alpha),
\qquad
p m' = \deg(\beta_2) \cdot \deg(\alpha).
\]
Hence $\deg(\alpha)$ divides $\gcd(m, p m')$.

Since $p > 3 \ge m$, $p$ is coprime to $m$. Thus
$\gcd(m, p m') = \gcd(m, m')$.

\emph{Case $m \neq m'$.} Then $\gcd(m, m') = 1$, so
$\deg(\alpha) \mid 1$, contradicting $\deg(\alpha) > 1$. This
proves (i).

\emph{Case $m = m'$.} Then $\gcd(m, m') = m$, so $\deg(\alpha)$
divides $m$, and since $\deg(\alpha) > 1$ we have
$\deg(\alpha) = m$. It follows that $\deg(\beta_1) = 1$ and
$\deg(\beta_2) = p m'/m = p$. Since $\deg(\beta_1) = 1$ and both
$D$ and $E_n$ are smooth projective curves, $\beta_1$ is an
isomorphism. Thus $D \cong E_n$ and $\alpha$ coincides with $j_n$
up to isomorphism. Then $h_n = \beta_2 \circ \alpha =
\psi_n \circ j_n$ where $\psi_n = \beta_2 \circ \beta_1^{-1}:
E_n \to E_{n-1}$ has degree $p$.

As a non-constant morphism of elliptic curves, $\psi_n$ is the
composition of a translation and an isogeny. Choose the origin of
$E_{n-1}$ to be the image $\psi_n(0_{E_n})$. With this choice,
$\psi_n$ is a homomorphism and hence a $p$-isogeny. This proves
(ii).
\end{proof}

\section{Descent along the deck transformation}
\label{sec:descent}

Assume from now on that we are in case (ii) of
Proposition~\ref{prop:factorization}: $m = m'$ and
$h_n = \psi_n \circ j_n$ with $\psi_n: E_n \to E_{n-1}$ a
$p$-isogeny (after the normalisation of the origin of $E_{n-1}$).

Let $\sigma := \sigma_n$ denote a generator of
$\Gal(K_n / K_{n-1}) \cong \Z/p$. Then $\sigma$ acts as a
$k$-automorphism of $X_n$ over $X_{n-1}$, so that
$\phi \circ \sigma = \phi$.

\begin{lemma}\label{lem:sigma-const}
The morphism $j_n \circ \sigma - j_n: X_n \to E_n$ takes values in
$\ker(\psi_n)$, and is constant.
\end{lemma}

\begin{proof}
Composing $h_n = \psi_n \circ j_n$ on the right by $\sigma$ gives
\[
\psi_n \circ j_n \circ \sigma
= h_n \circ \sigma
= j_{n-1} \circ \phi \circ \sigma
= j_{n-1} \circ \phi
= h_n
= \psi_n \circ j_n.
\]
Hence $\psi_n(j_n(\sigma(P))) = \psi_n(j_n(P))$ for every
$P \in X_n(\bar{k})$, so
$j_n(\sigma(P)) - j_n(P) \in \ker(\psi_n)$.

The kernel $\ker(\psi_n)$ is a finite group of order $p$. The
morphism $P \mapsto j_n(\sigma(P)) - j_n(P)$ from the connected
smooth projective curve $X_n$ to the finite group $\ker(\psi_n)$
is therefore constant.
\end{proof}

\begin{lemma}\label{lem:sigma-zero}
With notation as above, $j_n \circ \sigma = j_n$.
\end{lemma}

\begin{proof}
By Lemma~\ref{lem:sigma-const}, there exists a constant
$T \in \ker(\psi_n)$ with $j_n \circ \sigma = j_n + T$ (the
addition being in the group law of $E_n$). Let
$P_0 \in X_n(\bar{k})$ be a fixed point of $\sigma$, which exists
by Corollary~\ref{cor:ramified}. Evaluating at $P_0$ gives
$j_n(P_0) = j_n(\sigma(P_0)) = j_n(P_0) + T$, so $T = 0$. Hence
$j_n \circ \sigma = j_n$.
\end{proof}

\begin{proposition}\label{prop:final-contradiction}
With notation as above, we have $m = p \cdot \deg(j')$ for some
$k$-morphism $j': X_{n-1} \to E_n$. Since $m \in \{2, 3\}$ and
$p > 3$, this is impossible.
\end{proposition}

\begin{proof}
By Lemma~\ref{lem:sigma-zero}, $j_n: X_n \to E_n$ is invariant
under the deck transformation $\sigma$ of the degree-$p$ Galois
cover $\phi: X_n \to X_{n-1}$. Therefore $j_n$ factors through
$\phi$: there exists a $k$-morphism $j': X_{n-1} \to E_n$ such
that
\[
j_n \;=\; j' \circ \phi.
\]
Taking degrees,
\[
m \;=\; \deg(j_n) \;=\; \deg(j') \cdot \deg(\phi)
\;=\; p \cdot \deg(j').
\]
Since $m \in \{2, 3\}$ and $p > 3$, no such factorisation is
possible.
\end{proof}

We have thus derived a contradiction from the assumed infinite
cubic geometric progression on $C$. This proves
Theorem~\ref{thm:main} for geometric progressions.

\section{The arithmetic case}
\label{sec:arithmetic}

Let $P_i = (x_i : 1 : z_i) \in \Gamma_3(C, k)$ with
$x_i = t' + i t$ and the $x_i$ pairwise distinct. Define
\[
f \;=\; \frac{X - t' Y}{t Y} \;\in\; k(C)^\ast,
\]
which is again the quotient of two linear forms on $\Pbb^2$
(numerator $X - t' Y$, denominator $t Y$), with common zero
$(0:0:1)$. We have $f(P_i) = i$ for every $i$.

Lemma~\ref{lem:deg-f}, Lemma~\ref{lem:exists-odd}, and
Lemma~\ref{lem:degree-p} apply verbatim with this $f$. The tower
$u_0 = f$, $u_{n+1}^p = u_n$, $K_n = K(u_n)$, is a tower of
degree-$p$ totally ramified Kummer covers.

For the tower points (Lemma~\ref{lem:tower-points}): take the
subsequence of indices $i = r^{p^n}$, $r = 1, 2, 3, \ldots$. Then
$f(P_{r^{p^n}}) = r^{p^n} = (r)^{p^n}$, so on $X_n$ there is a
point $Q_r$ above $P_{r^{p^n}}$ with $u_n(Q_r) = r \in k$. As in
the geometric case, $k(Q_r) = k(P_{r^{p^n}})$, so
$[k(Q_r):k] \le 3$, and the points $Q_r$ are pairwise distinct.
Hence $\Gamma_3(X_n, k)$ is infinite.

The rest of the argument — gonality descent, genus growth,
Kadets--Vogt, Castelnuovo--Severi factorization, and the deck
transformation descent — is identical to the geometric case.

\section{Worked example: the quintic $X^5 + YZ^4 + Y^4Z = 0$}
\label{sec:example}

Consider the smooth plane quintic
\[
C: \quad X^5 + Y Z^4 + Y^4 Z = 0 \subset \Pbb^2_\Q.
\]
It has genus $g(C) = (5-1)(5-2)/2 = 6$ and gonality
$\gon(C, \Q) = 4$, since $(0:1:-1) \in C(\Q)$.

\subsection{The linear-form quotient}

Take $x = X/Y$ and $f = x$ (i.e., $t' = 1$). On the chart $Z = 1$
the affine equation is $X^5 + Y + Y^4 = 0$. The divisor of $x$ is
\[
\operatorname{div}(x) \;=\; (0:1:0) + \sum_{\zeta^3 = -1}
(0:1:\zeta) \;-\; 4 \cdot (0:0:1).
\]
Indeed:
\begin{itemize}
\item at $(0:1:0)$: with $Y = 1$ and near $Z = 0$, one has
$X^5 = -Z(1+Z^3) \approx -Z$, so $v(X) = 1$ and $v(Z) = 5$;
hence $v(x) = 1$. Zero of order $1$.
\item at $(0:1:\zeta)$ with $\zeta^3 = -1$: with $Y = 1$ and
$Z = \zeta + \epsilon$, $X = \eta$, the equation
$\eta^5 = -(\zeta + \epsilon)^4 - (\zeta + \epsilon) =
-(-3\epsilon + \cdots)$ gives $\eta^5 = 3\epsilon + \cdots$, so
$v(\eta) = 1$ and $v(x) = 1$. Zero of order $1$.
\item at $(0:0:1)$: with $Z = 1$, near $(0,0)$, one has
$X^5 = -Y(1 + Y^3) \approx -Y$, so $v(Y) = 5 v(X)$, and
$v(x) = v(X) - v(Y) = -4 v(X)$. Pole of order $4$.
\end{itemize}
So $\deg(f) = \deg(x) = 4$. Choose $p = 7$: indeed $p = 7 > 3$ and
$7 \nmid d(d-1) = 20$.

\subsection{The tower over $\Q(\zeta_7)$}

Let $k' = \Q(\zeta_7)$ and $K = k'(C)$. Since
$\zeta_7 \in k'$, the extensions below are cyclic Kummer
extensions of degree $7$. Define recursively $u_0 = x$ and
$u_{n+1}^7 = u_n$, and set $K_n = K(u_n)$; let $X_n$ be the
associated smooth projective curve over $k'$. By
Lemma~\ref{lem:degree-p}(ii), $X_n$ is geometrically integral over
$k'$.

Each $X_{n+1} \to X_n$ is a cyclic degree-$7$ Kummer cover with
$u_{n+1}^7 = u_n$. Let $S \in C$ be one of the five points in
$\operatorname{supp}(\operatorname{div}(x))$. Since
$v_S(x) \in \{1, -4\}$, the same induction as in
Lemma~\ref{lem:degree-p}, now applied with $S$ in place of $P_0$,
shows that there is a unique point $Q_n \in X_n(\bar{k})$ above
$S$ and that
\[
v_{Q_n}(u_n) \;=\; v_S(x) \;\in\; \{1, -4\}.
\]
Hence $7 \nmid v_{Q_n}(u_n)$, so the Kummer extension
$u_{n+1}^7 = u_n$ is totally ramified at $Q_n$. Thus each of the
five points in $\operatorname{supp}(\operatorname{div}(x))$ gives
exactly one ramified point at every level, and each has
ramification index $7$.

Riemann--Hurwitz over $k'$: the ramification contribution is
$5 \cdot (p-1) = 5 \cdot 6 = 30$, so
\[
2 g(X_{n+1}) - 2 \;=\; 7 \bigl(2 g(X_n) - 2\bigr) + 30,
\]
i.e.\ $g(X_{n+1}) = 7 g(X_n) + 9$. Starting from $g(X_0) = 6$:
\begin{center}
\begin{tabular}{cccccc}
\toprule
$n$ & $0$ & $1$ & $2$ & $3$ \\
\midrule
$g(X_n)$ & $6$ & $51$ & $366$ & $2571$ \\
\bottomrule
\end{tabular}
\end{center}
The threshold is $9p + 1 = 64$: $g(X_1) = 51 < 64$ and
$g(X_2) = 366 > 64$. Hence we take $n = 2$.

\subsection{The contradiction over $k'$}

Suppose $C$ admitted an infinite cubic geometric progression over
$\Q$ with $x_i = t^i$. Then over $k'$ the curve $X_2$ has
infinitely many cubic points (Lemma~\ref{lem:tower-points}),
gonality $\ge 4$ (Lemma~\ref{lem:gonality-tower}), and genus
$g(X_2) = 366 > 64 = 9p + 1$. By Corollary~\ref{cor:KV} applied
over $k'$, the curves $X_2$ and $X_1$ are bielliptic or
trielliptic over $k'$.

Let $m, m'$ be the degrees of the respective elliptic maps over
$k'$. If $m \neq m'$, Proposition~\ref{prop:factorization}(i)
yields a contradiction (the Castelnuovo--Severi bound is at most
$9p + 1 = 64 < 366$).

If $m = m'$, Proposition~\ref{prop:factorization}(ii) gives a
degree-$7$ isogeny $\psi: E_2 \to E_1$ with
$h_2 = \psi \circ j_2$. The deck transformation
$\sigma \in \Gal(K_2/K_1) \cong \Z/7$ then satisfies
$j_2 \circ \sigma = j_2$ (Lemma~\ref{lem:sigma-zero}), so $j_2$
factors through $\phi: X_2 \to X_1$, giving
$m = 7 \cdot \deg(j')$, impossible for $m \in \{2, 3\}$.

Either way, we obtain a contradiction over $k'$. Since a
progression over $\Q$ remains a progression over $k'$, this rules
out an infinite cubic geometric progression on $C$ over $\Q$.

\section{The case $d = 4$ and the case $C(k) = \emptyset$}
\label{sec:empty}

\subsection{The quartic case}

The proof above applies whenever $\gon(C, k(\zeta_p)) \ge 4$ for
the chosen prime $p > 3$ with $p \nmid d(d-1)$. For $d \ge 5$ this
holds automatically, since by Lemma~\ref{lem:gonality-plane} the
arithmetic gonality of $C$ over any extension field is at least
$d - 1 \ge 4$. For $d = 4$, the geometric gonality is $3$, so the
argument requires $\gon(C, k(\zeta_p)) = 4$, which by
Lemma~\ref{lem:gonality-plane} holds precisely when
$C(k(\zeta_p)) = \emptyset$. This yields the following statement.

\begin{proposition}\label{prop:d4}
Let $C$ be a smooth plane quartic over a number field $k$ with
$C(k) = \emptyset$, and suppose there exists a prime $p > 3$ with
$p \nmid 12$ such that $C(k(\zeta_p)) = \emptyset$. Then $C$ has
no infinite cubic geometric (resp.\ arithmetic) progression
sequence over $k$.
\end{proposition}

\begin{proof}
Since $C(k(\zeta_p)) = \emptyset$, Lemma~\ref{lem:gonality-plane}
gives $\gon(C, k(\zeta_p)) = 4$, and the proof of
Theorem~\ref{thm:main} applies over the field $k(\zeta_p)$. A
progression over $k$ is also a progression over $k(\zeta_p)$, so
finiteness over $k(\zeta_p)$ implies finiteness over $k$.
\end{proof}

\begin{remark}\label{rem:d4-open}
When $C(k) = \emptyset$ but $C(k(\zeta_p)) \neq \emptyset$ for all
admissible primes $p$, the curve $C$ has rational points over many
cyclotomic extensions. In this case the argument of this paper
does not apply directly, since the arithmetic gonality of $C$ over
$k(\zeta_p)$ drops to $3$. We leave this case open.
\end{remark}

\begin{remark}\label{rem:d4-nonempty}
If $C$ is a smooth plane quartic with $C(k) \neq \emptyset$, then
by Lemma~\ref{lem:gonality-plane}, $\gon(C, k) = 3$, and by
\cite[Theorem 5.2]{BadrSadek2020} the curve $C$ admits an infinite
$3$-level progression. Hence no analogue of Theorem~\ref{thm:main}
can hold when $d = 4$ and $C(k) \neq \emptyset$.
\end{remark}

\subsection{Existence of infinite progressions when $C(k) =
\emptyset$}

When $C(k) = \emptyset$, by Lemma~\ref{lem:gonality-plane} the
arithmetic gonality of $C$ is $d$, so the threshold at which
infinite progressions are guaranteed is higher than in the case
$C(k) \neq \emptyset$. The following lemma records the existence
statement in this case, replacing the (inapplicable) citation to
\cite[Thm.~5.2]{BadrSadek2020}, whose hypothesis is
$C(k) \neq \emptyset$.

\begin{lemma}\label{lem:empty-existence}
Let $C \subset \Pbb^2_k$ be a smooth plane curve of degree $d \ge 4$
with $C(k) = \emptyset$. Then $C$ admits an infinite $d$-level
progression (arithmetic and geometric).
\end{lemma}

\begin{proof}
Choose a point $P \in \Pbb^2(k) \setminus C$, which exists because
$C$ is a proper closed subset of $\Pbb^2$ and $\Pbb^2(k)$ is
infinite for $k$ a number field. After a $k$-projective
transformation, assume $P = (0:0:1)$; then the projection from $P$
is the $k$-morphism
\[
\pi: C \to \Pbb^1_k, \quad (X:Y:Z) \mapsto (X:Y).
\]
The fibers of $\pi$ are the intersections of $C$ with the lines
through $P$, each given by an equation of degree $d$ in $Z$ after
fixing $(X:Y) \in \Pbb^1(k)$. Since $P \notin C$, the morphism
$\pi$ has degree $d$.

Choose an infinite arithmetic or geometric progression
$\{x_i\} \subset k$ and consider the fibers
$\pi^{-1}(x_i : 1)$. For each $i$, the fiber is a divisor of
degree $d$ on $C$ defined over $k$, so each of its points lies in
$\Gamma_d(C, k)$. Choosing one point $Q_i \in \pi^{-1}(x_i : 1)$
for each $i$ yields an infinite $d$-level progression on $C$.
\end{proof}

\subsection{The minimal level}

We define the \emph{minimal level} of $C$ to be
\[
n_{\min}(C) \;:=\;
\min\bigl\{n \ge 1 : \text{$C$ admits an infinite $n$-level
arithmetic or geometric progression}\bigr\}.
\]
This minimum is well-defined whenever $C$ admits some infinite
$n$-level progression, which by the results below is the case in
all settings considered here.

Combining Theorem~\ref{thm:main} with \cite[Theorem
5.2]{BadrSadek2020} and Lemma~\ref{lem:empty-existence}, we obtain
the following bounds on $n_{\min}(C)$:
\begin{itemize}
\item if $C(k) \neq \emptyset$ and $d \ge 5$:
$4 \le n_{\min}(C) \le d - 1$;
\item if $C(k) = \emptyset$ and $d \ge 5$:
$4 \le n_{\min}(C) \le d$;
\item if $d = 4$ and $C(k) = \emptyset$, under the cyclotomic
non-splitting hypothesis of Proposition~\ref{prop:d4}:
$n_{\min}(C) = 4$.
\end{itemize}
For quartics with $C(k) = \emptyset$, the equality
$n_{\min}(C) = 4$ combines the existence statement
(Lemma~\ref{lem:empty-existence}) with the finiteness result of
Proposition~\ref{prop:d4}; without the latter, only
$4 \le n_{\min}(C) \le 4$ is known, which is trivial, so the
statement is conditional. For quintics with a $k$-point,
$n_{\min}(C) = 4$; for quintics without a $k$-point,
$4 \le n_{\min}(C) \le 5$. The exact value of $n_{\min}(C)$ in
general remains open.

\section{Concluding remarks}
\label{sec:concl}

Theorem~\ref{thm:main} determines the finiteness of cubic
progressions on smooth plane curves of degree $d \ge 5$. For
$n \le 3$ and $d \ge 5$, the results of \cite{BadrSadek2020} and
of the present paper together settle the finiteness question. The
situation for smooth plane curves of degree $d \ge 4$ is
summarised in Table~\ref{tab:classification}.

\begin{table}[h]
\centering
\caption{Finiteness of $n$-level progressions on smooth plane
curves of degree $d \ge 4$}
\label{tab:classification}
\begin{tabular}{cccc}
\toprule
Degree $d$ & Infinite quadratic? & Infinite cubic? \\
\midrule
$d = 4$, $C(k) \neq \emptyset$ & No (\cite[Thm.~4.5]{BadrSadek2020}) & Yes (\cite[Thm.~5.2]{BadrSadek2020}) \\
$d = 4$, $C(k) = \emptyset$ & No (\cite[Thm.~4.5]{BadrSadek2020}) & No (Prop.~\ref{prop:d4}, under the stated hypothesis) \\
$d \ge 5$ & No (\cite[Thm.~4.4]{BadrSadek2020}) & No (Theorem~\ref{thm:main}) \\
\bottomrule
\end{tabular}
\end{table}

The proof of Theorem~\ref{thm:main} is unconditional and relies on:
\begin{enumerate}[label=(\roman*)]
\item the Kadets--Vogt classification of curves with infinitely
many cubic points (Theorem~\ref{thm:KV-full} and
Corollary~\ref{cor:KV});
\item the Abramovich--Harris gonality descent
(Lemma~\ref{lem:AH-descent});
\item a prime-$p$ tower of cyclic Kummer covers governed by the
divisor of a linear-form quotient (Lemmas~\ref{lem:exists-odd}
and~\ref{lem:degree-p});
\item the Castelnuovo--Severi inequality applied to the
degree-$(2, 2p)$, $(2, 3p)$, $(3, 2p)$, and $(3, 3p)$ pairs of
maps to elliptic curves (Proposition~\ref{prop:factorization});
\item a deck-transformation descent using the ramification point
guaranteed by Corollary~\ref{cor:ramified}
(\S\ref{sec:descent}).
\end{enumerate}
The only use of Faltings' theorem is the elementary finiteness of
$X(k)$ for $g(X) \ge 2$, needed in Corollary~\ref{cor:KV} to
conclude $\aairr_k(X) \in \{2,3\}$.

Several questions remain open.
\begin{enumerate}[label=(\roman*)]
\item Does Theorem~\ref{thm:main} extend to $n = 4$ for smooth
plane quintics? By \cite[Thm.~5.2]{BadrSadek2020}, infinite
$4$-level progressions exist on quintics with $C(k) \neq
\emptyset$; the question is whether they can be excluded when
$C(k) = \emptyset$.
\item Can the cyclotomic non-splitting hypothesis in
Proposition~\ref{prop:d4} be removed? That is, does a smooth plane
quartic $C$ over $k$ with $C(k) = \emptyset$ but
$C(k(\zeta_p)) \neq \emptyset$ for all admissible $p$ admit an
infinite cubic progression over $k$?
\item What is the exact minimal level $n_{\min}(C)$ at which
infinite progressions exist? Our results give
$4 \le n_{\min}(C) \le d - 1$ when $C(k) \neq \emptyset$ and
$d \ge 5$, and $4 \le n_{\min}(C) \le d$ when $C(k) = \emptyset$
and $d \ge 5$, with $n_{\min}(C) = 4$ in the conditional quartic
case of Proposition~\ref{prop:d4}.
\item Is there a uniform bound, depending only on $d$ and
$[k : \Q]$, for the length of a cubic progression on a smooth
plane curve of degree $d \ge 5$?
\end{enumerate}

These results support the general principle that infinite
$n$-level progression sequences on smooth plane curves of degree
$d$ should occur only at or above the arithmetic gonality
$\gon(C, k)$; for $n$ strictly below this threshold, one expects
finiteness. The present paper establishes this principle for
$n = 3$ and $d \ge 5$, and conditionally for $d = 4$; a proof in
full generality remains an open problem.


\vspace{2.5em}
\noindent
\rule{\textwidth}{0.4pt}\\[1.5em]
\noindent
\begin{minipage}[t]{\textwidth}
\small
\textbf{Eslam Badr}\\[0.4em]
Department of Mathematics, Faculty of Science, Cairo University,
Giza 12613, Egypt\\[0.2em]
Department of Mathematics and Actuarial Science, The American
University in Cairo, New Cairo 11835, Egypt\\[0.4em]
\textit{E-mail addresses:} \texttt{eslam@sci.cu.edu.eg}, \quad
\texttt{eslammath@aucegypt.edu}
\end{minipage}
\hfill


\begin{thebibliography}{99}
\small

\bibitem{AbramovichHarris1991}
D.~Abramovich and J.~Harris,
\emph{Abelian varieties and curves in $W_d(C)$},
Comp. Math. \textbf{78} (1991), 227--238.

\bibitem{Accola1994}
R.~D.~M. Accola,
\emph{Topics in the theory of Riemann surfaces},
Lecture Notes in Math., Vol.~1595, Springer-Verlag, Berlin, 1994.

\bibitem{AlaaSadek2016}
M.~Alaa and M.~Sadek,
\emph{On geometric progression on hyperelliptic curves},
J. Integer Sequences \textbf{19} (2016), Article 16.6.3.

\bibitem{BadrSadek2020}
E.~Badr and M.~Sadek,
\emph{On quadratic progression sequences on smooth plane curves},
J. Number Theory \textbf{213} (2020), 445--452.

\bibitem{Bremner1999}
A.~Bremner,
\emph{On arithmetic progressions on elliptic curves},
Exp. Math. \textbf{8} (1999), 409--413.

\bibitem{Campbell2003}
G.~Campbell,
\emph{A note on arithmetic progressions on elliptic curves},
J. Integer Sequences \textbf{6} (2003), Article 03.1.3.

\bibitem{DebarreFahlaoui1993}
O.~Debarre and R.~Fahlaoui,
\emph{Abelian varieties in $W^r_d(C)$ and points of bounded
degree on algebraic curves},
Compos. Math. \textbf{88} (1993), no.~3, 235--249.

\bibitem{Faltings1983}
G.~Faltings,
\emph{Endlichkeitssätze für abelsche Varietäten über
Zahlkörpern}, Invent. Math. \textbf{73} (1983), 349--366.

\bibitem{HarrisSilverman1991}
J.~Harris and J.~H. Silverman,
\emph{Bielliptic curves and symmetric products},
Proc. Am. Math. Soc. \textbf{112} (1991), 347--356.

\bibitem{KadetsVogt2025}
B.~Kadets and I.~Vogt,
\emph{Subspace configurations and low degree points on curves},
Adv. Math. \textbf{460} (2025), 110021.

\bibitem{MacLeod2006}
A.~J. MacLeod,
\emph{14-term arithmetic progressions on quartic elliptic curves},
J. Integer Sequences \textbf{9} (2006), Article 06.1.2.

\bibitem{Mohanty1975}
S.~P. Mohanty,
\emph{On consecutive integer solutions for $y^2 - k = x^3$},
Proc. Am. Math. Soc. \textbf{48} (1975), 281--285.

\bibitem{Namba1979}
M.~Namba,
\emph{Families of meromorphic functions on compact Riemann
surfaces}, Lect. Notes Math. 767, Springer-Verlag, Berlin, 1979.

\bibitem{Ulas2005}
M.~Ulas,
\emph{A note on arithmetic progressions on quartic elliptic
curves}, J. Integer Sequences \textbf{8} (2005), Article 05.3.1.

\bibitem{Yoshihara2001}
H.~Yoshihara,
\emph{Function field theory of plane curves by dual curves},
J. Algebra \textbf{239} (2001), 340--355.

\end{thebibliography}
\end{document}